\documentclass[11pt]{amsart}
\usepackage{graphicx}
\usepackage{latexsym}
\usepackage{fancyhdr}
\usepackage{amsmath, amssymb, stmaryrd}
\usepackage[all]{xy}
\usepackage[french, english]{babel}
\usepackage{pdflscape}
\usepackage{longtable}
\usepackage{rotating}
\usepackage{hyperref}
\usepackage{subfigure}
\usepackage{tensor}
\usepackage{enumerate}

\usepackage{color}

\theoremstyle{plain}
\newtheorem{thm}{Theorem}[section]

\newtheorem{lemma}[thm]{Lemma} 
\newtheorem{prop}[thm]{Proposition}

\theoremstyle{definition}
\newtheorem{defi}[thm]{Definition}
\theoremstyle{remark}
\newtheorem{rem}[thm]{Remark}
\numberwithin{equation}{section}

\newtheorem{ex}[thm]{Example}

\newcommand{\lgw}{\longrightarrow}
\newcommand{\lgm}{\longmapsto}

\newcommand{\lb}{\llbracket}
\newcommand{\rb}{\rrbracket}

\newcommand{\ovl}{\overline}

\newcommand{\wdh}{\widehat}

\newcommand{\wdt}{\widetilde}

\newcommand{\R}{\mathbb{R}}
\newcommand{\K}{\mathbb{K}}

\newcommand{\N}{\mathbb{N}}
\newcommand{\C}{\mathbb{C}}

\newcommand{\Q}{\mathbb{Q}}
\renewcommand{\t}{\tau}

\renewcommand{\a}{\alpha}

\renewcommand{\b}{\beta}

\renewcommand{\phi}{\varphi}

\newcommand{\w}{\underline{w}}
\newcommand{\e}{\varepsilon}

\begin{document}
\title{Artin approximation over Banach spaces}

\author{Guillaume Rond}
\email{guillaume.rond@univ-amu.fr}
\address{Universit\'e Publique, France}

\begin{abstract}
We give examples showing that the usual Artin Approximation theorems valid for convergent  series over a field are no longer true for convergent  series over a commutative Banach algebra. In particular we construct an example of a commutative integral  Banach algebra $R$ such that the ring of formal power series over $R$ is not flat over the ring of convergent power series over $R$.
\end{abstract}

\thanks{The author is deeply grateful to the UMI LASOL of the CNRS where this project has been carried out.}

\subjclass[2010]{13J05, 13B40, 16W80, 46J99}
\keywords{Banach algebra, flatness, Artin approximation}

\maketitle
%
%%%%%

\section{Introduction}
The classical Artin Approximation Theorem is the following:

\begin{thm}\label{Ar_cl}\cite{Ar68}
  Let $F(x,y)$ be a  
  vector of convergent power series over $\C$ in two sets of variables $x$ and $y$. Assume given a formal power series solution $\wdh{y}(x)$ vanishing at 0,
  $$F(x,\wdh{y}(x))=0.$$
  Then, for any $c\in\N$,  there exists a convergent power series solution $y(x)$ vanishing at 0,
  $$F(x,y(x))=0$$
  which coincides with $\wdh{y}(x)$ up to degree $c$,
  $$y(x)\equiv \wdh{y}(x) \text{ modulo } (x)^c.$$
  
  \end{thm}

The main tools for proving this theorem are the implicit function theorem and the Weierstrass division theorem. But in the case the equations $F(x,y)$ are linear in $y$, this theorem is equivalent to the faithful flatness of the morphism $\C\{x\}\lgw \C\lb x\rb$ (see \cite[Example 1.4]{R} for instance or \cite[I. ¤3 Proposition 13]{B}). In fact the faithful flatness of this morphism comes from the fact that $\C\{x\}$ is a Noetherian local ring. And the Noetherianity of $\C\{x\}$ is usually proved by using the Weierstrass division theorem.\\
Another version of this theorem is the following one:

\begin{thm}\label{SArtin}\cite{Ar69}\cite{W}
  Let $F(x,y)$ be a vector of convergent power series over $\C$ in two sets of variables $x$ and $y$. Then  for any integer $c$ there exists an integer $\b$ such that for any  given approximate solution  $\ovl{y}(x)$ at order $\b$, $\ovl y(0)=0$,
  $$F(x,\ovl{y}(x))\equiv 0\text{ modulo } (x)^{\b},$$
 there exists a formal power series solution $y(x)$ vanishing at 0,
  $$F(x,y(x))=0$$
  which coincides with $\ovl{y}(x)$ up to degree $c$,
  $$y(x)\equiv \ovl{y}(x) \text{ modulo } (x)^c.$$
  \end{thm}
  In particular this result implies that, if $F(x,y)=0$ has approximate solutions at any order, then it has a formal (even convergent by Theorem \ref{Ar_cl}) power series solution.\\
  Let us mention that these results remain valid when we replace $\C$ by a complete valued field, or when we replace the ring of convergent power series over $\C$ by the ring of algebraic power series over a field. In fact these results remain true in the more general setting of excellent Henselian local rings by \cite{P} (see \cite{R} for a review of all these different results).\\
  \\
  The aim of this note is to show that these results are no longer true when we replace $\C$ by a commutative Banach algebra over $\R$ or $\C$. In the first part we construct a commutative Banach algebra $R$ such that $R\{t\}\lgw R\lb t\rb$ is not  flat, showing that Artin approximation theorem is not true for linear equations with coefficients in $R\{t\}$.\\
    Let us mention here that  $R\lb t\rb$ is flat over $R$, for a commutative ring $R$,  if and only if $R$ is coherent (indeed $R\lb t\rb$ is a direct product of copies of $R$ - see \cite[Theorem 2.1]{C}). And there are several known examples of Banach algebras which are not coherent (in fact most of the known Banach algebras are not coherent; see for instance \cite{MVR} or  \cite{H} and the references herein). But the flatness of $R\{t\}\lgw R\lb t\rb$ is a different property that is not related to the coherence of $R$.\\
   In the second part we provide an example of one polynomial $F(y)$ with coefficients in  $R[t]$, where $R$ is the Banach algebra of holomorphic functions over a disc, with the following property: $F(y)$ has approximate solutions up to any order but has no solution in $R\lb t\rb$. This shows that Theorem \ref{SArtin} does not hold for convergent power series over a Banach algebra. Let us mention that this example is a slight modification of an example of Spivakovsky related to a similar problem \cite{S}.\\
   Nevertheless we mention that  in the case where $R$ is a complete valuation ring of rank one (in particular a non-archimedean Banach algebra), Schoutens and Moret-Bailly proved several extensions of Theorems \ref{Ar_cl} and \ref{SArtin} (see \cite{Sc} and \cite{MB}).\\
   The note has been motivated by questions from Nefton Pali and Wei Xia.
   
%\section*{Acknowledgment}

%%%%%%%%%%%%%%%%%%%%%%%%%%%%%%%%%%%%%%%%%%%%%%%%%%%%%%%%%%%%%%%%%%%%%%%%%%%%%%%%%%%%%%%%%%%%%
%%%%%%%%%%%%%%%%%%%%%%%%%%%%%%%%%%%%%%%%%%%%%%%%%%%%%%%%%%%%%%%%%%%%%%%%%%%%%%%%%%%%%%%%%%%%%
%%%%%%%%%%%%%%%%%%%%%%%%%%%%%%%%%%%%%%%%%%%%%%%%%%%%%%%%%%%%%%%%%%%%%%%%%%%%%%%%%%%%%%%%%%%%%
\section{A Banach algebra $R$ such that $R\{t\}\lgw R\lb t\rb $ is not  flat}

Let $\K=\R$ or $\C$. We begin by the following definition of power series in countable many indeterminates:
\begin{defi}
Let $\N^{(\N)}$ be the submonoid of $\N^\N$ formed by the sequences whose all but finitely terms are 0. Let $(x_i)_{i\in\N}$ be a countable family of indeterminates. Then  $\K\lb x_i\rb_{i\in\N}$ is the set of series $\sum_{\a\in \N^{(\N)}} a_\a x^\a$
where $x^\a=x_0^{\a_0}\cdots x_n^{\a_n}\cdots$. This former product is finite since $\a_i=0$ for $i$ large enough. This set is a commutative ring since the sum of sequences $\N^{(\N)}\times\N^{(\N)}\lgw \N^{(\N)}$ has finite fibers (see \cite[Chapter III, $\S$ 2, 11]{Bo}). Let us mention that this ring is not the $(x)$-adic completion of $\K[x]$, the ring of polynomials in the $x_i$ (see \cite{Y} for instance).
\end{defi}

 Let $x$, $y$, $z$ and $w_k$ for $k\in\N$ be indeterminates. For simplicity we denote by $\w$ the vector of indeterminates $(w_0,w_1,\ldots)$. We denote by $\K[x,y,z,\w]$ the ring of polynomials in the indeterminates $x$, $y$, $z$, $\w$.\\
 For a polynomial $\displaystyle p=\sum_{k\in\N,l\in\N,m\in\N,\a\in  \N^{(\N)}} a_{k,l,m,\a}x^ky^lz^m\w^\a \in\K[x,y,z,\w]$ we set
$$\|p\|:=\sum_{k,l,m,\a} |a_{k,l,m,\a}|.$$
This is well defined because the sum is finite.
This defines a norm on $\K[x,y,z,\w]$. \\
We denote by $\K\{x,y,z,\w\}$ the completion of $\K[x,y,z,\w]$ for this norm. This is the following commutative Banach algebra:
$$\left\{\sum_{k\in\N,l\in\N,m\in\N,\a\in \N^{(\N)}}a_{k,l,m,\a}x^ky^lz^m\w^\a \mid \sum_{k\in\N,l\in\N,m\in\N,\a\in \N^{(\N)}}|a_{k,l,m,\a}|<\infty\right\}$$
and the norm of an element $\displaystyle f:=\sum_{k\in\N,l\in\N,m\in\N\a\in  \N^{(\N)}}a_{k,l,m,\a}x^ky^lz^m\w^\a$ is
$$\| f\|:=\sum_{k\in\N,l\in\N,m\in\N,\a\in  \N^{(\N)}}|a_{k,l,m,\a}|.$$
In particular $\K\{x,y,z,\w\}$ is a subring of $\K\lb x,y,z, w_i\rb_{i\in\N}$.\\
We denote by $I$ the ideal  of $\K[x,y,z,\w]$ generated by the polynomials
$$xw_0-z^2 \text{ and } yw_k-(k+1)xw_{k+1} \text{ for all } k\geq 0.$$
The ideal $I\K\{x,y,z,\w\}$ is not closed since it is not finitely generated. Thus, we denote by $\ovl I$ its closure. This is the set of sums
$$\sum_{k\in\N} f_k(x,y,z,\w)$$
such that $f_k(x,y,z,\w)\in I\K\{x,y,z,\w\}$ and $\sum_k\|f_k(x,y,z,\w)\|<\infty$.\\
\begin{defi}
 We denote by $R$ the Banach $\K$-algebra $\K\{x,y,z,\w\}/\ovl I$. 
 \end{defi}
 In order to denote that two series $f$ and $g\in\K\{x,y,z,\w\}$ have the same image in $R$, we write $f\equiv_R g$. The norm of the image $\ovl f$ of an element $f\in\K\{x,y,z,\w\}$ is
 $$\|\ovl f\|=\inf_{g\in \ovl I}\|f+g\|=\inf_{g\in I}\|f+g\|.$$
Now we denote by $R\{t\}$ the ring of convergent series in the indeterminate $t$ with coefficients in $R$. We have the following result:

\begin{prop}\label{thm1}
The linear equation
\begin{equation}\label{main_eq} (x-yt)f(t)=z^2\end{equation}
has a unique solution $f(t)$ in $R\lb t\rb$ and this solution is not convergent.

\end{prop}

From this we will deduce  the following result:
\begin{thm}\label{main_thm}
The Banach $\K$-algebra $R$ is an integral domain and 
 the morphism $R\{t\}\lgw R\lb t\rb$ is not flat.
\end{thm}

%%%%%%%%%%%%%%%%%%%%%%%%%%%%%%%%%%%%%%%%%%%%%%%%%%%%%%%%%%%%%%%%%%%%%%%%%%%%%%%%%%%%%%%%%%%%%%%
\subsection{Proofs of Proposition \ref{thm1} and Theorem \ref{main_thm}}
We begin by giving the following key result:
\begin{lemma}\label{canc}
$x$ is not a zero divisor in $R$.

%We have
%$$(0:x)_{R}=(0).$$
\end{lemma}

%%%%%
\begin{proof}
First of all, we will determine a subset  of $\K\{x,y,z,\w\}$ such that  every element of $\K\{x,y,z,\w\}$ is equal modulo $\ovl I$  to a unique series of this subset.\\
 First  we remark that 
\begin{equation}\label{op1}M_1:=yw_iw_j\equiv_R (i+1)xw_{i+1}w_j\equiv_R \frac{i+1}{j}yw_{i+1}w_{j-1}=:M_2\end{equation}
for all integers $i$ and $j$ with $i<j$. 
If $j=i+1$ these two monomials are equal, otherwise the largest index of a monomial $w_j$ appearing in the expression of $M_2$ is strictly less than for $M_1$.\\
Now we have, for $i>0$:
\begin{equation}\label{op2}z^2w_i\equiv_R xw_0w_i\equiv_R\frac{1}{i}yw_0w_{i-1}.\end{equation}
A well chosen composition of these operations transforms any monomial of the form $Cx^az^ky^lw_0^{n_0}\cdots w_i^{n_i}$ into a monomial of the form $rCx^{a'}z^{k'}y^lw_0^{n'_0}\cdots w_j^{n'_j}$ where $j$ is minimal and $r\in(0,1]$.\\% Hence these two operations replace elements of $\K\{x,y,z,\w\}$ by elements of $\K\{x,y,z,\w\}$.\\
By repeating these two operations we may reduce every monomial to a constant times one of the following monic monomials:
\begin{equation}\label{cases}\left\{\begin{array}{cc} 
z^{\e}x^ay^l w_0^{n_0}& \text{ with } a>1,  l, n_0\geq 0\text{ and } \e\in\{0,1\},\\
z^{\e}y^lw_i^{n_i} & \text{ with }l>0, i>0, n_i>0  \text{ and } \e\in\{0,1\},\\
z^{\e}y^lw_i^{n_i}w_{i+1}^{n_{i+1}} &\text{ with }l>0,i\geq 0, n_i,n_{i+1}>0  \text{ and } \e\in\{0,1\},\\
 z^{\e} w_0^{n_0}\ldots w_{i}^{n_i}& \text{ with } n_i>0\text{ with }\e\in\{0,1\},\\
%z^ky^lw_0^{n_0}& \text{ with }k\geq 0, l> 0, n_0\geq 0.
\end{array}\right.\end{equation}
We denote by $E$ the subset of $\K[x,y,z,\w]$ of polynomials that are sums of monomials of \eqref{cases} (up to multiplicative constants), and by $\ovl{E}$ the closure of $E$ in $\K\{x,y,z,\w\}$, that is the set of convergent power series whose non zero monomials are those of \eqref{cases} (up to multiplicative constants). We have shown that every polynomial is equivalent to a polynomial of $E$ modulo $I$. To prove the unicity we proceed as follows.\\
We set 
$$F_0:=xw_0-z^2, F_{k+1}:=yw_k-(k+1)xw_{k+1} \ \text{ for } k\geq 0$$
$$G_{k,l}:=(l+1)yw_kw_{l+1}-(k+1)yw_lw_{k+1} \ \text{ for all } k<l.$$
Then we consider the following monomial order: We define
$$x^ay^kz^lw_1^{\a_1}\cdots w_n^{\a_n}>x^{a'}y^{k'}z^{l'}w_1^{\a_1'}\cdots w_n^{\a_n'}$$
if 
$$a+k+l+\sum_i\a_i>a'+k'+l'+\sum_i\a_i',\ \text{ 
or }
a+k+l+\sum_i\a_i=a'+k'+l'+\sum_i\a_i' $$ 

$$ \text{ and } (l,a,k, \a_n,\ldots, \a_{0})>_{lex}(l',a',k',\a_n',\ldots, \a_0')$$
where $>_{lex}$ denotes the lexicographic order. That is, we first compare the total degree of two monomials, then we order the indeterminates as
$$z>x>y>w_l>w_k \ \text{ for all } l>k.$$
We claim that $\{F_j,G_{k,l}\}_{j,k,l\in\N,\ l>k}$ is a Gr\"obner basis of $I$ for this order. In order to prove this, we only need to compute the S-polynomials of the elements of this set of polynomials, and then their reduction (see \cite{CLO} for the terminology). This is Buchberger's Algorithm which is very classical in the Noetherian case. The case of polynomial rings in countably many indeterminates works identically, cf. \cite[Proposition 1.13]{IY} for instance. The only S-polynomials we have to consider are those of polynomials whose leading terms are not coprime, that is, for $l>k$, 
$$S(F_{k+1},F_{l+1});\ S(G_{k,l},F_{l+1});\ S(G_{k,l},F_{k}).$$
We have
$S(F_{k+1},F_{l+1})=G_{k,l}.$
Moreover
$$S(G_{k,l},F_{l+1})=y(yw_kw_l-(k+1)xw_lw_{k+1}).$$
This leading term of $S(G_{k,l},F_{l+1})$ is $-(k+1)xyw_lw_{k+1}$, and it is equal to $y(F_{k+1}w_l-yw_kw_l)$. Therefore
$S(G_{k,l},F_{l+1})=F_{k+1}yw_l.$\\
Finally we have 
$$S(G_{k,l},F_{k})=kx((l+1)yw_kw_{l+1}-(k+1)yw_lw_{k+1})+(l+1)yw_{l+1}(yw_{k-1}-kxw_{k})$$
$$=(l+1)y^2w_{k-1}w_{l+1}-k(k+1)xyw_lw_{k+1}.$$
Its leading term  is $-k(k+1)xyw_lw_{k+1}$ and it is divisible by the leading term of $F_{k+1}$. The remainder of the division of $S(G_{k,l},F_{k})$ by $F_{k+1}$ is
$$ (l+1)y^2w_{k-1}w_{l+1}-ky^2w_kw_l=yG_{k-1,l}$$
Therefore the reductions of these S-polynomials is always zero, hence the family $\{F_j,G_{k,l}\}_{j,k,l\in\N,\ l>k}$ is a Gr\"obner basis of $I$. Thus, the initial ideal of $I$ is generated by the monomials
$$z^2, xw_{k+1}, yw_kw_{l+1} \text{ for } 0\leq k< l.$$
Therefore every polynomial of $\K[x,y,z,\w]$ is equivalent modulo $I$ to a unique polynomial of $E$. \\
Now let $f\in\K\{x,y,z,\w\}$. We can write $f=\sum_{n\in\N} C_nx^{a_n}y^{b_n}z^{c_n}\w^{\a_n}$ where the $C_n$ are in $\K^*$. In particular $\sum_n|C_n|<\infty$. For every $n\in\N$, there is a unique $(a'_n,b'_n,c'_n,\a'_n)$ and a unique $r_n\in(0,1]$ such that 
$$C_nx^{a_n}y^{b_n}z^{c_n}\w^{\a_n}-r_nC_nx^{a'_n}y^{b'_n}z^{c'_n}\w^{\a'_n}\in I$$
and $x^{a_n}y^{b_n}z^{c_n}\w^{\a_n}$ has one the forms given in \eqref{cases}. Now, for every $n\in\N$, we set
$$g_n:=\sum_{k=0}^{n-1} r_kC_kx^{a'_k}y^{b'_k}z^{c'_k}\w^{\a'_k}+\sum_{k\geq n}C_kx^{a_k}y^{b_k}z^{c_k}\w^{\a_k}.$$
In particular we have that $P_n:=f-g_n\in I$ and the sequence $(g_n)_n$ converges in $\K\{x,y,z,\w\}$ to the series $g=\sum_{k\in\N} r_kC_kx^{a'_k}y^{b'_k}z^{c'_k}\w^{\a'_k}\in\K\{x,y,z,\w\}$. Therefore the sequence $(P_n)_n$ converges in $\K\{x,y,z,\w\}$, and its limits is in $\ovl I$. \\
Therefore, every power series of $\K\{ x,y,z,\w\}$ can be  written as a  sum of a power series in $\ovl I$ and a convergent power series whose monomials are as in \eqref{cases} (up to multiplicative constants). \\
We remark that, by repeating  \eqref{op1} $\lfloor \frac{j-i}{2}\rfloor$ times, we have 
$$yw_iw_j\equiv_R  ryw_{i+\lfloor \frac{j-i}{2}\rfloor}w_{j-\lfloor \frac{j-i}{2}\rfloor}$$
for some constant $r\in(0,1]$. Moreover applying \eqref{op2} reduces by 2 the degree in $z$ of a monomial. Therefore, a monomial of the form
$$Cx^ay^bz^c w_1^{\a_1}\cdots w_j^{\a_j}$$
of total degree $d=a+b+c+\sum_k\a_k$, is not equal to a monomial involving only the indeterminates
$$x,y,z, \text{ and } w_i \text{ for } i<\frac{j-\frac{c}{2}}{2^d}.$$
Moreover \eqref{op1} and \eqref{op2} transforms monomials into monomials of the same degree since $I$ is generated by homogeneous binomials. Therefore, given a monomial $M$ among those of \eqref{cases} (up to some multiplicative constant), there is finitely many monomials that are equal to $M$ modulo $I$.\\% In particular if $P\in\K[x,y,z,\w]$ is equal to 0 in $R$, then it is equal to 0 in $\K[x,y,z,\w]$ modulo $I$. \\
Now let $f\in\ovl E\cap \ovl I$, $f=\sum_{(a,b,c,\a)} f_{(a,b,c,\a)}x^ay^bz^c\w^\a$. Let us fix such $(a,b,c,\a)$ such that $x^ay^bz^c\w^\a$ is one of the monic monomials of \eqref{cases}. There is only a finite number of distinct monomials that are equal to $f_{(a,b,c,\a)}x^ay^bz^c\w^\a$ modulo $I$. Let us denote them by
$$C_1x^{a_1}y^{b_1}z^{c_1}\w^{\a_1},\ldots, C_Nx^{a_N}y^{b_N}z^{c_N}\w^{\a_N}.$$
We can remark that there is only a finite number of $F_l$  that have a monomial that divides at least one of the following monic monomials
\begin{equation}\label{list}x^ay^bz^c\w^\a, x^{a_1}y^{b_1}z^{c_1}\w^{\a_1},\ldots, x^{a_N}y^{b_N}z^{c_N}\w^{\a_N}.\end{equation}
We denote them by
$F_{l_1},\ldots, F_{l_p}.$
Because $f\in\ovl I$, we can write
$f=\sum_{l\in\N} f_lF_l$
where the $f_l$  are in $\K\lb x,y,z,\w\rb$. For every $i\in\{1,\ldots, p\}$  we remove from $f_{l_i}$  all the monomials that do not divide one of the monomials \eqref{list}, and we denote by $f'_{l_i}$  the resulting polynomial. Then we have that
$$P:=\sum_{i=1}^p f'_{l_i}F_{l_i}\in I.$$
By construction the coefficients of the monomials \eqref{list} in the expansion of $P$ are the corresponding coefficients in the expansion of $f$, that is   
$$f_{(a,b,c,\a)},0,\ldots, 0$$
respectively. Therefore, the coefficient of $x^ay^bz^c\w^\a$ in the expansion of the unique $Q\in E$ such that $Q\equiv_R P$, is equal to $f_{(a,b,c,\a)}$ because no other monomial than those listed in \eqref{list} (up to some multilplicative constants) is equivalent to a monomial of the form $Cx^ay^bz^c\w^\a$ where $C\in\K^*$.
But $Q=0$ since $P\in I$, thus  $f_{(a,b,c,\a)}=0$. Hence $f=0$ and $\ovl E\cap\ovl I=0$.\\
Therefore every series of $\K\{x,y,z,\w\}$ is equivalent modulo $\ovl I$ to a unique series of $\ovl E$.\\
\\
Now take $f\in\K\{x,y,z,\w\}$ such that $x\equiv_R0$. We can write $f=xp(x,y,z,w_0)+q(y,z,\w)$ and assume that the monomials in the expansion of $xp(x,y,z,w_0)+q(y,z,\w)$ are only those of  \eqref{cases}.
Then
$$x^2p(x,y,z,w_0)+xq(y,z,\w)\equiv_{ R} 0.$$
The representation of $x^2p(x,y,z,w_0)+xq(y,z,\w)$ as a sum of monomials as in \eqref{cases} has the form
\begin{equation}\label{zero}x^2p(x,y,z,w_0)+xq(y,z,w_0,0)+\ovl q(y,z,\w)=0\end{equation}
where $\ovl q(y,z,\w)$ is the series obtained from $xq(y,z,\w)-xq(y,z,w_0,0)$ by replacing the monomials as follows (using the two previous operations \eqref{op1} and \eqref{op2}):
\begin{equation}\label{cases_map}\left\{\begin{array}{ccc} xz^{\e}y^lw_i^{n_i} &\lgm & \frac{1}{i}z^\e y^{l+1}w_{i-1}w_i^{n_i-1}, \text{if } i>0\\
xz^{\e}y^lw_i^{n_i}w_{i+1}^{n_{i+1}} & \lgm& \frac{1}{i+1}z^\e y^{l+1}w_i^{n_i+1}w_{i+1}^{n_{i+1}-1}, \text{ if } i>0\\
 xz^{\e} w_0^{n_0}\ldots w_{i}^{n_i}& \lgm   & Cz^\e yw_j^{m_j}w_{j+1}^{m_{j+1}} \text{ or } Cz^\e yw_{j}^{m_{j}}\\
\text{ for } i>0 \text{ and }n_i>0 &   & \text{ for some }C\in\K, | C|\leq 1, j\geq 0 \\
%xz^\e y^lw_0^{n_0}& \lgm & xz^{\e}y^l w_0^{n_0}
\end{array}\right.\end{equation}
Indeed for the third monomial we have 
$$ xz^{\e} w_0^{n_0}\ldots w_{i}^{n_i}\equiv_R \frac{1}{i+1}z^\e y w_0^{n_0}\cdots w_{i-1}^{n_{i-1}+1}w_i^{n_{i}-1} $$
and this monomial on the right side can be transformed into  a monomial of the form $Cz^\e yw_j^{m_j}w_{j+1}^{m_{i+1}}$ or $Cz^\e yw_j^{m_j}$ for some $C\in\K$, $| C|\leq 1$, and $j\geq 0$, by using the two  operations  \eqref{op1} and \eqref{op2} on monomials.\\
 This shows that the three types of monomials that we obtain after multiplication by $x$ are all distinct, that is the map defined by \eqref{cases_map} is injective. By \eqref{zero} we have $\ovl q(y,z,\w)=0$, therefore $q(y,z,\w)-q(y,z,w_0,0)=0$.\\
Moreover, again by \eqref{zero}, we have
 $$x^2p(x,y,z,w_0)+xq(y,z,w_0,0)=0.$$
 This shows that $x^2p(x,y,z,w_0)+xq(y,z,\w)=0$. Therefore $x$ is not a zero divisor in $R$.
 \end{proof}
%%%%

\begin{proof}[Proof of Proposition \ref{thm1}]
Let $f(t)\in R\lb t\rb $ such that
$$(x-yt)f(t)=z^2.$$
By writing $f=\sum_{k=0}^\infty f_kt^k$ with $f_k\in R$ for every $k$, we have
$$xf_0=z^2$$
$$xf_k-yf_{k-1}=0\ \ \forall k\geq 1.$$
Thus
$$xf_0=z^2=xw_0$$
so $x(f_0-w_0)=0$ and $f_0=w_0$ by Lemma \ref{canc}. Then we will prove by induction on $k$ that $f_k=k!\,w_k$ for every $k$. Assume that this is true for an integer $k\geq 0$. Then we have
$$xf_{k+1}=yf_k=k!\,yw_k=(k+1)!\,xw_{k+1}.$$
Hence $x(f_{k+1}-(k+1)!\,w_{k+1})=0$ and $f_{k+1}=(k+1)!\,w_{k+1}$ by Lemma \ref{canc}.
Therefore the only solution of
$$(x-yt)f(t)=z^2$$
is the series $\sum_{k=0}^\infty k!\, w_kt^k$, and this one is  divergent because $\|w_k\|=1$. This holds because in every element of $I$, the monomial $w_k$ has coefficient 0.

\end{proof}

Now we can give the proof of Theorem \ref{main_thm}:

\begin{proof}[Proof of Theorem \ref{main_thm}]
 Since $x$ is not a zero divisor in $R$ by Lemma \ref{canc}, the localization morphism
$$R\lgw R_{1/x}$$
is injective. But $R_{1/x}$ is isomorphic to $\K\{x,y,z\}_{1/x}$ since in $R_{1/x}$ we have
$$w_0=z^2/x\text{ and } \forall k\geq 0, w_k=\frac{1}{k!}y^kz^2x^{k+1}.$$
But $\K\{x,y,z\}_{1/x}$ is an integral domain (this is a localization of the integral domain $\K\{x,y,z\}$), therefore so is $R$.\\
\\
Now assume that  the morphism $R\{t\}\lgw R\lb t\rb$ is flat. By \cite[Theorem 7.6]{M} applied to the linear equation $(x-yt)F-z^2G=0$, there exist an integer $s\geq 1$, and convergent series
$$a_1(t),\ldots, a_s(t), b_1(t),\ldots, b_s(t)\in R\{t\}$$
such that 
\begin{equation}\label{eq6}(x-yt)a_i(t)-z^2b_i(t)=0 \text{ for every }i,\end{equation}
and formal power series
$$h_1(t),\ldots, h_s(t)\in R\lb t\rb$$ such that
$$f(t)=\sum_{i=1}^s a_i(t)h_i(t),\ 1=\sum_{i=1}^sb_i(t)h_i(t).$$
Indeed the vector $(f(t),1)$ is a solution of the linear equation
$$(x-yt)f(t)-z^2g(t)=0$$
with $f(t):=\sum_{k=0}^\infty k!\, w_kt^k$.\\
 Then
$$\wdt g(t):=\sum_{i=1}^sb_i(t) h_i(0)=1+t\e(t)$$
for some $\e(t)\in R\{t\}$. Since $1$ is a unit of $R$, $1+t\e(t)$ is a unit in $R\{t\}$. \\
Set $\wdt f(t):=\sum_i a_i(t) h_i(0)$. By \eqref{eq6}, $(\wdt f(t),\wdt g(t))$ is a solution of the equation
$$(x-yt)\wdt f(t)-z^2\wdt g(t)=0.$$
Since $\wdt g(t)$ is a unit in $R\{t\}$ we have
$$(x-yt)\wdt f(t)\wdt g(t)^{-1}=z^2.$$
This contradicts Theorem \ref{thm1}. Therefore $R\{t\}\lgw R\lb t\rb$ is not flat.

\end{proof}
%%%%%%%%%%%%%%%%%%%%%%%%%%%%%%%%%%%%
%%%%%%%%%%%%%%%%%%%%%%%%%%%%%%%%%%%%
%%%%%%%%%%%%%%%%%%%%%%%%%%%%%%%%%%%%

\section{An Example concerning the strong Artin approximation theorem}
Let $n$ be a positive integer, $x=(x_1,\ldots, x_n)$ and $\rho>0$. We set $\K=\R$ or $\C$. Then
$$B^n_\rho:=\left\{f=\sum_{\a\in\N^n}a_\a x^\a\mid ||f||_\rho:=\sum_{\a\in\N^n}|a_\a|\rho^{|\a|}<\infty\right\}$$
is a Banach space equipped with the norm $||\cdot||_\rho$. Of course $\K[x]\subset B^n_\rho$.\\

\begin{rem}
We do not have 
$$B^n_\rho\lb t\rb \cap \K\{x,t\}=B^n_\rho\{t\}.$$
For instance, the power series
$$f=\sum_{k\in\N}x_1^{k!}t^k$$
is a convergent power series in $(x,t)$, belongs to $B^n_2\lb t\rb $, but
$$\sum_k\|x_1^{k!}\|_2\t^k=\sum_k2^{k!}\t^k=\infty$$
for every $\t>0$. Therefore $f\notin B^n_2\{t\}$.
\end{rem}

We provide two examples based on an example of Spivakovsky concerning the extension of Theorem \ref{SArtin} to the nested case (see \cite{S}).

%%%%%%%%%%%%%%%%%%%%%%%%%%%%%%%%%%%%
\begin{ex}\label{ex1} Let $n=1$ and set
$$F(x,t,y_1,y_2):=xy_1^2-(x+t)y_2^2\in B_\rho\{t\}[y_1,y_2].$$
Let $$\sqrt{1+t}=1+\sum_{n\geq 1}a_nt^n\in\Q\{ t\}$$ be the unique power series  such that $(\sqrt{1+t})^2=1+t$ and whose value at the origin is 1.
For every $c\in\N$ we  set $y_2^{(c)}(t):=x^c$ and $y_1^{(c)}(t):=x^c+\sum_{n=1}^ca_nx^{c-n}t^n\in B_\rho\{t\}$. Then 
$$F(x,t,y_1^{(c)}(t),y_2^{(c)}(t))\in (t)^{c+1}.$$
On the other hand the equation  $f(x,t,y_1(t),y_2(t))=0$ has no  solution $(y_1(t)$, $y_2(t))\in B_\rho\{t\}^2$ but $(0,0)$. Indeed let us denote by $T_0$ the Taylor map at 0:
$$T_0:B_\rho\{t\}\lgw \K\lb x,t\rb.$$
If $f(x,t,y_1(t),y_2(t))=0$ then 
$$xT_0(y_1(t))^2-(x+t)T_0(y_2(t))^2=0.$$
But since $\K\lb x,t\rb$ is a unique factorization domain, this equality implies that $T_0(y_1(t))=T_0(y_2(t))=0$, hence $y_1(t)=y_2(t)=0$.
\\
\\
This shows that there is no $\b:\N\lgw \N$ such that for every $y(t)\in B_\rho\{t\}^2$ and every $k\in\N$ with
$$F(x,t,y(t))\in (t)^{\b(k)}$$
there exists $\wdt y(t)\in B_\rho\{t\}^2$ such that
$$F(x,t,\wdt y(t))=0$$
 and $\wdt y(t)-y(t)\in (t)^k$.\\
 \end{ex}
 %%%%%%%%%%%%%%%%%%%%%%%%%%%%%%%%
 
\begin{ex} We can modify a little bit the previous example to construct a $F$ as before that does not depend on $t$. We set
 $$G(x,y_1,y_2,y_3):=xy_1^2-(x+y_3)y_2^2\in B_\rho[y_1,y_2,y_3].$$
For every $c\in\N$ we  set $y_2^{(c)}(t):=x^c$, $y_1^{(c)}(t):=x^c+\sum_{n=1}^ca_nx^{c-n}t^n$ and $y_3^{(c)}(t):=t\in B_\rho\{t\}$. Then 
$$G(x,y_1^{(c)}(t),y_2^{(c)}(t),y_3^{(c)}(t))\in (t)^c.$$
Now if $\wdt y(t)\in B_\rho\{t\}^3$ satisfies $G(x,\wdt y(t))=0$ and
$$\wdt y(t)-y(t)\in (t)^2$$
then $\wdt y_3(t)=x+t+\e(t)$ with $\e(t)\in (t^2)$. Thus $x+\wdt y_3(t)$ is an irreducible  power series in $x$ and $t$, and it is coprime with $x$. By the same argument based on the Taylor map as in Example \ref{ex1},  the relation
$$x\wdt y_1(t)^2-(x+t+\e(t))\wdt y_2(t)^2=0$$
implies that $\wdt y_1(t)=\wdt y_2(t)=0$.\\
\\
This shows that there is no $\b:\N\lgw \N$ such that for every $y(t)\in B_\rho\{t\}^3$ and every $k\in\N$ with
$$G(x,y(t))\in (t)^{\b(k)}$$
there exists $\wdt y(t)\in B_\rho\{t\}^3$ such that
$$G(x,\wdt y(t))=0$$
 and $\wdt y(t)-y(t)\in (t)^k$.\end{ex}


\begin{thebibliography}{<99>}


\bibitem[{1}]{Ar68} M. Artin, On the solutions  of analytic equations, \emph{Invent. Math.}, \textbf{5}, (1968), 277-291.

\bibitem[{2}]{Ar69} M. Artin, Algebraic  approximation  of structures over complete local rings, \emph{Publ. Math. IHES}, \textbf{36}, (1969), 23-58.

\bibitem[3]{Bo} N. Bourbaki, Alg\`ebre,  Chapitres 1 \`a 3, Hermann, Paris, 1970.

\bibitem[4]{B} N. Bourbaki, Alg\`ebre commutative, Chapitres 1 \`a 4, Masson, Paris, 1985. 

\bibitem[5]{C}  S. U. Chase,  Direct products of modules, \emph{Trans. Amer. Math. Soc.}, \textbf{97}, (1960), 457-473.

\bibitem[6]{CLO} D. Cox, J. Little, and D. OÕShea, Ideals, varieties, and algorithms. UndergraduateTexts in Mathematics. Springer-Verlag, New York, second edition, 1997.

\bibitem[7]{IY} K. Iima, Y, Yoshino, Gr\"obner bases for the polynomial ring with infinite variables and their applications, \emph{Comm. Algebra}, \textbf{37}, (2009), no. 10, 3424-3437.

\bibitem[8]{H} M. Hickel, Noncoh\'erence de certains anneaux de fonctions holomorphes, \emph{Illinois J. Math.}, \textbf{34}, (1990), no. 3, 515-525.

\bibitem[9]{MVR} W. S. McVoy, L. A. Rubel, Coherence of some rings of functions, \emph{J. Func. Anal.}, \textbf{21}, (1976), 76-87.

\bibitem[{10}]{M} H. Matsumura, Commutative   Ring Theory,  \emph{Cambridge studies in advanced mathematics}, 1989.

\bibitem[11]{MB} L Moret-Bailly, An extension of Greenberg's theorem to general valuation rings, \emph{Manuscripta math.},  \textbf{139}, n¡ 1 (2012), 153-166.

\bibitem[12]{P} D. Popescu, General N\'eron desingularization and approximation, \emph{Nagoya Math. J.}, \textbf{104}, (1986), 85-115.


\bibitem[12]{Ri} P. Ribenboim, Fields: algebraically closed and others, \emph{Manuscripta Math.}, \textbf{75}, (1992), 115-150.


\bibitem[13]{R} G. Rond,  Artin Approximation, \emph{J. Singul.},  \textbf{17}, (2018), 108-192.

\bibitem[14]{Sc} H. Schoutens, Approximation properties for some non-noetherian local rings, \emph{Pacific J. Math.}, \textbf{131}, (1988), 331-359.


\bibitem[{15}]{S} M. Spivakovsky, Non-existence of the  Artin function for   Henselian pairs, \emph{Math. Ann.}, \textbf{299}, (1994), 727-729.

\bibitem[{16}]{W} J. J. Wavrik, A theorem on solutions of  analytic equations with applications to deformations of complex structures, \emph{Math. Ann.}, \textbf{216}, (1975), 127-142.

\bibitem[17]{Y} A. Yekutieli,  On flatness and completion for infinitely generated modules over Noetherian rings, \emph{Comm. Algebra}, \textbf{39}, (2011), 4221-4245.

\bibitem[18]{Y2} A. Yekutieli, Flatness and Completion Revisited, \emph{Algebras and Representation Theory}, \textbf{21}, Issue 4, (2018), 717-736.

\end{thebibliography}
\end{document}